\documentclass{article}

\usepackage{amsmath, amsfonts, amsthm,amssymb}
\usepackage{multirow}
\usepackage{graphicx}
\usepackage{cite}
\def\ind{{\rm 1\hspace{-0.90ex}1}}

\def\EE{\mathbb{E}}
\def\PP{\mathbb{P}}
\def\NN{\mathbb{N}}
\def\RR{\mathbb{R}}

\def\ind{{\rm 1\hspace{-0.90ex}1}}

\def\=d{\stackrel{d}{=}}

\newtheorem{theorem}{Theorem}

\newtheorem{lemma}[theorem]{Lemma}
\newtheorem{corollary}[theorem]{Corollary}

\newtheorem{condition}{Condition}

\begin{document}

\title{On Bootstrap Percolation in Living Neural Networks}
\author{ Hamed Amini \\ \small{\'Ecole Normale Sup\'{e}rieure - INRIA Rocquencourt, Paris, France}\\
\small{\it hamed.amini@ens.fr}}
\date{}
\maketitle

\begin{abstract}
 Recent experimental studies of living neural networks \cite{Eck07,brsomotl06} reveal that their global activation induced by electrical stimulation can be explained using the concept of bootstrap percolation on a directed random network. The experiment consists in activating externally an initial random fraction of the neurons and observe the process of firing until its equilibrium. The final portion of neurons that are active depends in a non linear way on the initial fraction. The main result of this paper is a theorem which enables us to find the asymptotic of final proportion of the fired neurons in the case of random directed graphs with given node degrees as the model for interacting network. This gives a rigorous mathematical proof of a phenomena observed by physicists in neural networks~\cite{Coh09}.
\end{abstract}

\section{Introduction}
Understanding the structure and  dynamics of neural networks is a challenge for biologists, physicists and mathematicians.
Recent experimental studies of living neural networks~\cite{Eck07,brsomotl06} reveal that their global activation induced by electrical stimulation can be explained using the concept of bootstrap percolation on a directed random network. The experiment consists in activating externally an initial random fraction of the neurons and observe the process of firing until its equilibrium. The final portion of neurons that are active depends in a non linear way on the initial fraction. The main result shown by experiments is that there exists a non-zero critical value for the fraction of initially (i.e., externally) excited neurons beyond which the global activity jumps to a almost complete activation of the network, while below this critical value the firing essentially does not spread. The main result of this paper is a theorem which enables us to find the asymptotic of final proportion of the fired neurons in the case of random directed graphs with given node degrees as the model for interacting network. This gives a rigorous mathematical proof of a phenomena observed by physicists in neural networks~\cite{Coh09}. Cohen et al. in ~\cite{Coh09} find this asymptotic via mean-field assumption and they compare it to simulations and experiment.  The validity of the random graph approximation to metric graphs such as the experimental neural networks is discussed in \cite{Tlusty09}. Bootstrap percolation model has been used in several related applications (see for example \cite{amdrle, goltsev-2006, watts, schwab-2008}). For a review, we refer the reader to \cite{AdL03}.

A neural network is a group of interconnected neurons functioning as a circuit. The neural network is modeled as a directed graph ~\cite{brsomotl06} whose nodes are neurons connected by synapses. The total number of neurons is $n$. Let $G =(V,E)$ be a directed graph on the vertex set $V=[1,...,n]$. We denote $i \rightarrow j$ if there is a directed link from $i$ to $j$. Each node has two degrees, an in-degree, the number of links that point into the node, and an out-degree, which is the number pointing out. We denote by $d_i^{+}$ the in-degree of node $i$ and $d_i^{-}$ its out-degree. Let $A$ denote the adjacency matrix of graph G, with $A_{ij}=1$ if $j \rightarrow i$ and $0$ otherwise. 
Once a neurons fired, it stays on forever. At the beginning of the process, a neuron has a probability $\alpha$ to fire as a direct response to the externally applied electrical stimulus and will be on at time $t+ 1$ if at time $t$ it was on, or if at least $\Omega$ of its incoming nodes were on at time $t$.

We denote by $X_i(t)$ the state of the neuron $i$ at time $t$. It is on if $X_i(t)=1$ and off if $X_i(t)=0$. Then at each time step $t+1$, each node $i$ applies:
\begin{eqnarray}
X_i(t+1)=X_i(t)+(1-X_i(t))\ind\left(\sum_j A_{ij}X_j(t) \geq \Omega \right),
\end{eqnarray}
where $\ind\left(\Xi \right)$ denotes the indicator of an event $\Xi$; this is $1$ if $\Xi$ holds and $0$ otherwise.
The dynamics is monotonic from the definition, since a firing neuron can never turn off and therefore $X_i(t+1) \geq X_i(t)$.
Let us define $\Phi_{n}(\alpha,t)$ as $$\Phi_{n}(\alpha,t):= n^{-1}\sum_{j=1}^n \EE[X_j(t)] .$$ We are interested to find the asymptotic value of $$\Phi_n(\alpha) :=  \lim_{t \rightarrow \infty} \Phi_{n}(\alpha,t),$$ when $n \rightarrow \infty$,
in the case of random directed graphs with arbitrary degree distribution (see for example Molloy and Reed ~\cite{molloyreed95, Molloy98thesize}, Janson \cite{janson08}, Newman, Strogatz and Watts\cite{newstrowat} and Cooper and Frieze\cite{coopfri04}) as the underlying model for the interacting network. Let us define $P(j,k)$ to be the probability that a randomly chosen vertex has in-degree $j$ and out-degree $k$. We remark that in general this joint distribution of $j$ and $k$ is not equal to the product $p_j p_k$ of the separate distributions of in- and out-degree. Since every edge on a directed graph must leave some vertex and enter another, $P(j,k)$ must satisfy
$\sum_{j,k}(j-k)P(j,k) = 0 .$ The next section describes this model of random digraphs.


\subsection{Notation and definitions}
We are interested in constructing a random directed graph on $n$ vertices. Let $\mbox{\textbf{d}}^+_n=\{(d_{n,i}^{+})^n_{i=1}\}$ and $\mbox{\textbf{d}}^-_n=\{(d_{n,i}^{-})^n_{i=1}\}$ be sequences of non-negative integers such that $\sum_{i=1}^n d_{n,i}^{+} = \sum_{i=1}^n d_{n,i}^{-}$. The configuration model (CM) on $n$ vertices with degree sequences $\mbox{\textbf{d}}^+_n$ and $\mbox{\textbf{d}}^-_n$ is constructed as follows:

 A vertex $i$ is represented by the set of its incoming and outgoing edges that we denote by $W_i^+$, and $W_i^-$ respectively with $|W_i^+|=d_i^+$, $|W_i^-|=d_i^-$. Let $W^+ = \bigcup_i W_i^+$ and $W^- = \bigcup_i W_i^-$.
A configuration is a matching of $W^+$ with $W^-$ and we choose the configuration at random, uniformly over all possible configurations.  We denote the resulted graph by $CM(n,\mbox{\textbf{d}}^+_n,\mbox{\textbf{d}}^-_n)$. Observe that the self-loops may occur, these become rare as $n \rightarrow \infty$ (see e.g. \cite{bollobas}, \cite{janson06} or \cite{coopfri04} for more precise results in this direction).

We will let $n\rightarrow \infty$, and assume that for each $n$, given $\mbox{\textbf{d}}^+_n$ and $\mbox{\textbf{d}}^-_n$ satisfying the following regularity conditions:

\begin{condition}
\label{cond}
For each $n \in \NN$, $\mbox{\textbf{d}}^+_n=\{(d_{n,i}^{+})^n_{i=1}\}$ and $\mbox{\textbf{d}}^-_n=\{(d_{n,i}^{-})^n_{i=1}\}$ are sequences of nonnegative integers such that $\sum_{i=1}^n d_{n,i}^{+} = \sum_{i=1}^n d_{n,i}^{-}$, and, for some probability distribution $P(j,k)$ independent of $n$,
\begin{enumerate}
\item $\# \{i: d_{n,i}^+=j, d_{n,i}^-=k\}/n \rightarrow P(j,k)$ as $n \rightarrow \infty$ (the degree density condition: the density of vertices of in-degree $j$ and out-degree $k$ tends to $P(j,k)$);
\item $\sum_ {j,k}j P(j,k) = \sum_ {j,k}k P(j,k)=:\lambda \in (0,\infty)$ (finite expectation property);
\item $\sum_{i=1}^n d_{n,i}^+/n = \sum_{i=1}^n d_{n,i}^-/n \rightarrow \lambda$ as $n \rightarrow \infty$ (the average degree tends to a given value $\lambda$).
\end{enumerate}
\end{condition}

We consider the asymptotic case when $n \rightarrow \infty$ and say that an event holds w.h.p. (with
high probability) if it holds with probability tending to 1 as $n \rightarrow \infty$. We shall use $\stackrel{p}{\rightarrow}$ for convergence in probability as $n \rightarrow \infty$. Similarly, we use $o_p$ and $O_p$ in a standard way. for example, if $(X_n)$ is a sequence of random variables, then $X_n = O_p(1)$ means that "$X_n$ is bounded in probability" and $X_n=o_p(n)$ means that $X_n/n \stackrel{p}{\rightarrow} 0$.

\subsection{Statement of result}
In this section, we state the main theorem of this work.
Let $D_{in}$ and $D_{out}$ be random variables with the distribution $\mathbb{P}(D_{in} = j,D_{out}=k) = P(j,k)$.
We define the function $f_{\alpha}(y)$ as follows
\begin{eqnarray*}
f_{\alpha}(y)&:=&\lambda y - (1-\alpha) \mathbb{E}\left[ D_{out} \ind\left(Bin(D_{in},1-y) < \Omega\right)\right] .
\end{eqnarray*}
Let $y^*=y^*_{\alpha}$ be the largest solution to $f_{\alpha}(y) = 0$ in $[0,1]$, i.e.,
$$y^* = \max \{y\in[0,1]|f_{\alpha}(y) = 0\} .$$ Remark that such $y^*$ exists because $f_{\alpha}(0) \leq 0$, $f_{\alpha}(1) = \lambda \alpha > 0$ and $f_{\alpha}$ is continuous.
The main result of this paper is the following theorem.

\begin{theorem}
\label{thm-main}
Consider the random graph $CM(n,\mbox{\textbf{d}}^+_n,\mbox{\textbf{d}}^-_n)$ satisfying Condition \ref{cond}. Then we have
\begin{enumerate}
\item If $y^*=0$, i.e., if $f_{\alpha,\theta}(y) > 0$ for all $y \in (0,1]$, then w.h.p. $$\Phi^{(n)}(\alpha) = 1 - o_p(1) .$$
\item If $y^*>0$ and furthermore $y^*$ is not a local minimum point of $f_{\alpha}(y)$, then w.h.p.
$$\Phi_n(\alpha) = 1 - (1-\alpha)\mathbb{E}\left[\ind\left(Bin(D_{in},1-y^*)<\Omega \right)\right] + o(1).$$
\end{enumerate}
\end{theorem}

\subsection{Simulation}
Following \cite{brsomotl06, Coh09, Sor08}, we assume a gaussian distribution
for in-degree $\PP(D_{in}=k) \sim exp \left(\frac{-(k-\overline{k})^2}{2\sigma^2} \right)$ with $\overline{k}=50$ and
$\sigma=15$ based on the experimental results \cite{brsomotl06, Sor08}.
By Theorem \ref{thm-main}, one could see that when $D_{in}$ and $D_{out}$ are independent, $\Phi(\alpha) := \lim_{n \to \infty} \Phi_n (\alpha)$ will depend
only on the distribution of in-degree $D_{in}$. Figure \ref{figNN} shows the three dimensional
representation of the final fraction of fired neurons, i.e. $\Phi(\alpha)$,  as a function of ${\alpha}$ and $\Omega$. For comparison of this with experiment, we refer to Cohen et al.~\cite{Coh09} (who showed that the model and experiment fit very well when using the gaussian distribution of the connections).

\begin{figure}[ht]
\centering
\includegraphics[height=7.4cm]{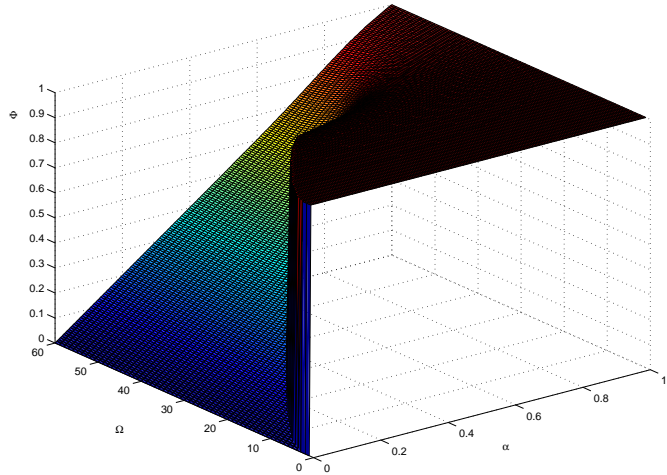}
\caption{
The final fraction of fired neurons as a function of ${\alpha}$ and $\Omega$. Here $\PP(D_{in}=k) \sim exp \left(\frac{-(k-\overline{k})^2}{2\sigma^2} \right)$ with $\overline{k}=50$ and
$\sigma=15$.}
\label{figNN}
\end{figure}
\subsection{Organization of the paper}
Bootstrap Percolation is studied in detail in the next section. We describe the dynamics of bootstrap percolation as a Markov chain in Section \ref{sec-MC}. The proof of out main theorem,  Theorem \ref{thm-main}, is based on the use of differential equations for solving discrete random processes. This was first introduced by Wormald~\cite{Worm95}. We briefly discuss his method in Section~\ref{sec-eqdiff}. The proof of our main result is given in Section~\ref{sec-proof}.



\section{Preliminaires}
\subsection{The Markov chain}\label{sec-MC}
The aim of this section is to describe the dynamics of bootstrap percolation as a Markov chain, which is perfectly tailored for asymptotic study. We consider the bootstrap percolation on $CM(n,\mbox{\textbf{d}}^+_n,\mbox{\textbf{d}}^-_n)$. Let $m(n):=\sum_{i=1}^n d^+_{n,i}$ denote the number of out-going edges in the graph. Our analysis below consists in an extension of the results of~\cite{balpit07}.

At a given time step $t$ neurons are partitioned into fired $\mathbb{F}(t)$ and non-fired $\mathbb{N}(t)$. We further partition the class of non-fired nodes according to their in and out degree $\mathbb{N}(t) = \bigcup_{j, k}\mathbb{N}^{j, k}(t)$. At time zero, $\mathbb{F}(0)$ contains the initial set of fired neurons.
We look at the system in discrete time. At time step $t+1$ we have
 $$ \mathbb{F}(t+1) = \mathbb{F}(t) \bigcup \left\{v \in \mathbb{N}(t) \ \ \mbox{such that} \ \ | \mathbb{F}(t) \bigcap \{w\in V, A_{vw}=1\} | \geq \Omega \right\} .$$

 We use a different approach based on bilateral interactions. This allows for a simpler Markov chain description of the system. At each step we have one interaction only between two neurons, yielding at least one fired.
Our process is as follows
 \begin{itemize}
 \item Choose an out-going edge of a fired neuron $i$.
 \item Identify its partner $j$ (i.e. by construction of the random graph in the configuration model, the partner is given by choosing an in-going edge randomly among all available in-going edges)
 \item Delete both edges. If $j$ is currently non-fired and it is the $\Omega$-th deleted in-going edge from $j$, then $j$ fires.
 \end{itemize}

 By definition, an interaction means coupling an out-going edge with an in-going edge.
 Our system is described in terms of
 \begin{itemize}
 \item $N_i^{j, k}(t)$, $0 \leq i < \Omega$, the number of non-fired neurons with in-degree $j$, out-degree $k$, and $i$ in-going edges from fired neurons at time $t$,
 \item $F^{j, k}(t)$ : the number of fired neurons with in-degree $j$ and out-degree $k$ at time $t$,
 \item $F(t)$: the number of fired neurons at time $t$,
 \item $N_{in}(t)$ : the number of  in-going edges belonging to non-fired neurons at time $t$,
 \item $F_{in}(t)$ : the number of  in-going edges belonging to fired neurons at $t$,
 \item $F_{out}(t)$ : the number of out-going edges belonging to fired neurons at $t$.
\end{itemize}
Obviously for $j>i$ we have $N_i^{j, k}(t)=0$. Because at each step we delete $1$ in-going edge and the number of in-going edges at time $0$ is $m(n)$, the number of existing in-going edges at time $t$ will be $m(n)-t$ and we have $$F_{in}(t)+N_{in}(t) = m(n) -t .$$
It is easy to see that the following identities hold:
\begin{eqnarray}
N_{in}(t) &=& \sum_{j,k}\sum_{i<\Omega} (j - i) N_i^{j, k}(t) , \\
F_{out}(t) &=& \sum_{k,j} k F^{j, k}(t) - t ,\\
F(t) &=& \sum_{k,j} F^{j, k}(t) .
\end{eqnarray}
The process will finish at the stopping time $T_f$ which is the first time $t\in \NN$ where $F_{out}(t)=0$. The final number of fired neurons will be $F(T_f)$.
By definition of our process
$\left(N_i^{j, k}, F^{j, k} \right)_{i, j, k} $ represents a Markov chain. We write the transition probabilities of the Markov chain.

There are three possibilities for the $B$, the partner of an out-going edge of a fired neuron $A$.
\begin{enumerate}
\item $B$ is fired, the next state is
\begin{eqnarray*}
N_i^{j, k}(t+1) &=& N_i^{j, k}(t) , \ \ (0 \leq i < \Omega) ,\\
F^{j, k}(t+1) &=& F^{j, k}(t).
\end{eqnarray*}

\item $B$ is non-fired of in-degree $j$, out-degree $k$ and this is the $(i+1)$-th deleted in-coming edge and $i+1 < \Omega$. The probability of this event is $\frac{(j - i) N_i^{j, k} }{m(n)-t}$. The next state is
\begin{eqnarray*}
N_i^{j, k}(t+1) &=& N_i^{j, k}(t) - 1 ,\\
N_{i+1}^{j, k}(t+1) &=& N_{i+1}^{j, k}(t) + 1 ,\\
N_{i'}^{j, k}(t+1) &=& N_{i'}^{j, k}(t) , \ \ (0 \leq i' < \Omega, \ i' \neq i,i+1) , \\
F^{j, k}(t+1) &=& F^{j, k}(t).
\end{eqnarray*}

\item $B$ is non-fired of in-degree $j$, out-degree $k$ and this is the $\Omega$-th deleted in-coming edge. Then $j \geq \Omega$ and with probability $\frac{(j - \Omega + 1) N_{\Omega-1}^{j, k}(t) }{m(n) -t}$, we have
\begin{eqnarray*}
N_i^{j, k}(t+1) &=& N_i^{j, k}(t) , \ \ (0 \leq i<\Omega-1) ,\\
N_{\Omega-1}^{j, k}(t+1) &=& N_{\Omega-1}^{j, k}(t) - 1 , \\
F^{j, k}(t+1) &=& F^{j, k}(t) + 1.
\end{eqnarray*}
\end{enumerate}

Let $P_t$ denote the pairing generated by time $t$, i.e., $P_t=\{e_{out},e_{in}\}$ be the set of edges picked before time $t$. We obtain the following equations for expectation of $\left( N_i^{j, k}(t+1) , F^{j, k}(t+1)\right)$ conditioned on $P_t$ by averaging over the possible transitions:

\begin{eqnarray*}
\EE\left[N_0^{j, k}(t+1) - N_0^{j, k}(t) | P_t \right] &=& - \frac{j N_0^{j, k}(t)}{m(n) -t} ,\\
\EE\left[N_i^{j, k}(t+1) - N_i^{j, k}(t) | P_t \right] &=&  \frac{(j-i+1) N_{i-1}^{j, k}(t)-(j-i) N_i^{j, k}(t)}{m(n)-t},  (0<i<\Omega),\\
\EE\left[F^{j, k}(t+1) - F^{j, k}(t) | P_t \right] &=&  \frac{(j-\Omega+1) N_{\Omega-1}^{j, k}(t)}{m(n)-t} .
\end{eqnarray*}

\subsection{Wormald's Theorem}\label{sec-eqdiff}
In this section we briefly present a method introduced by Wormald in~\cite{Worm95} for the analysis of a discrete random process by using
differential equations. In particular we recall a general purpose theorem for the use of this method.
This method has been used to analyze several kinds of algorithms on random graphs and random regular graphs (e.g., ~\cite{CainWorm05}, \cite{Worm99} and \cite{Molloy98thesize}).

Recall that a function $f(u_1, . . . , u_j)$ satisfies a Lipschitz condition on $D \in \RR^j$ if a constant $L > 0$ exists with the property that
$$|f(u_1, . . . , u_j)-f(v_1, . . . , v_j)| \leq L \max _{1\leq i \leq j}|u_i-v_i|$$ for all $(u_1, ..., u_j)$ and $(v_1, ...., v_j)$ in $D$. For variables $Y_1, ..., Y_b$ and for $D \in \RR^{b+1}$,
the stopping time $T_D(Y_1, ..., Y_b)$ is defined to be the minimum $t$ such that $(t/n; Y_1(t)/n, ..., Y_b(t)/n) \notin D$. This is written as $T_D$ when $Y_1, ..., Y_b$ are understood from the context.

The following theorem is a reformulation of Theorem 5.1 of~\cite{Worm99}, modified and extended for the case of an infinite number of variables. In it, "uniformly" refers to the convergence implicit in the $o()$ terms. Hypothesis $(1)$ ensures that $Y_t$ does not change too quickly throughout the process. Hypothesis $(2)$ tells us what we expect for the rate of change to be, and property $(3)$ ensures that this rate does not change too quickly.

\begin{theorem}[Wormald \cite{Worm99}]
\label{thm-eqdif1}
Let $b=b(n)$ be given ($b$ is the number of variables). For $1 \leq l \leq b$, suppose $Y_l(t)$ is a sequence of real-valued random variables, such that $0 \leq Y_l(t) \leq C n$ for some constant $C$, and $H_t$ be the history of the sequence, i.e. the sequence $\left\{Y_j(k), \ 0 \leq j \leq b, \ 0 \leq k \leq t\right\}$.

Suppose also that for some bounded connected open set $D = D(n) \subseteq \RR^{b+1}$ containing the intersection of $\{(t,z_1,...,z_b): t\geq 0\}$ with some neighborhood of
$$\{(0,z_1,...,z_b):\PP(Y_l(0)=z_l n, 1\leq l\leq b) \neq 0 \mbox{ for some n}\},$$
the following three conditions are verified:
\begin{enumerate}
  \item {\rm(Boundedness).} For some function $\beta=\beta(n) \geq 1$ and for all $t<T_D$
  $$\max_{1 \leq l \leq b} |Y_l(t+1)-Y_l(t)|\leq \beta $$ 
  \item {\rm(Trend).} For some function $\lambda=\lambda_1(n) = o(1)$ and for all $l \leq b$ and $t<T_D$
  $$|\EE[Y_l(t+1)-Y_l(t)|H_{t}]-f_l(t/n,Y_1(t)/n,...,Y_l(t)/n)| \leq \lambda_1 .$$
  \item {\rm(Lipschitz).}  For each $l$ the function $f_l$ is continuous and satisfies a Lipschitz condition on $D$ with all Lipschitz constants uniformly bounded.
\end{enumerate}
Then the following holds
\begin{description}
  \item[(a)] For $(0,\hat{z}_1,...,\hat{z}_b) \in D$, the system of differential equations
  $$\frac{dz_l}{ds}=f_l(s,z_1,...,z_l), \ \ l=1,...,b ,$$ has a unique solution in $D$, $z_l:\RR \rightarrow \RR$ for $l=1,\dots,b$, which passes through $z_l(0)=\hat{z}_l,$ $l=1,\dots,b,$, and which extends to points arbitrarily close to the boundary of $D$.
  \item[(b)] Let $\lambda>\lambda_1$ with $\lambda=o(1)$.
  For a sufficiently large constant C, with probability $1-O\left(\frac{b\beta}{\lambda} \exp \left( - \frac{n\lambda^3}{\beta^3}\right)\right)$, we have
 $$Y_l(t)=n z_l(t/n) + O(\lambda n)$$ uniformly for $0 \leq t \leq \sigma n$ and for each $l$. Here $z_l(t)$ is the solution in (a) with $\hat{z}_l = Y_l(0)/n$, and $\sigma = \sigma(n)$ is the supremum of those $s$ to which the solution can be extended  before reaching within $l^{\infty}$-distance $C\lambda$ of the boundary of $D$.
\end{description}
\end{theorem}

We will also use the following corollary of the above theorem, which is namely Theorem
6.1 of \cite{Worm99}. This theorem states that, as long as condition $3$ holds in $D$, the
solution of the system of equations above can be extended beyond the boundary of $\hat{D}$,
into $D$.

\begin{corollary}
\label{cor-eqdif2}
For any set $\hat{D}=\hat{D}(n) \subseteq \RR ^{b+1}$, let $T_{\hat{D}}=T_{\hat{D}(n)}(Y_1,...,Y_b)$ be the minimum $t$ such that
$(\frac tn, \frac{Y_1(t)}n,\dots,\frac{Y_b(t)}n) \notin \hat{D}$ (the stopping time). Assume in addition
that the first two hypotheses of Theorem~\ref{thm-eqdif1} are verified but only within the restricted range $t < T_{\hat{D}}$ of
$t$. Then the conclusions of the theorem hold as before, after replacing $0 \leq t \leq \sigma n$ by  $0 \leq t \leq \min\{\sigma n,T_{\hat{D}}\}$.
\end{corollary}
\begin{proof}
For $1 \leq j \leq b$ , define random variables $\hat{Y}_j$ by
$$\hat{Y}_j(t+1) = \left\{
    \begin{array}{ll}
        Y_j(t+1) & \mbox{if} \ t < T_{\hat{D}} \\
        Y_j(t) + f_j(t/n,Y_1(t)/n,...,Y_j(t)/n) & \mbox{otherwise}
    \end{array}
\right.$$ for all $t \geq 0$. Then the $\hat{Y}_j$ satisfy the hypotheses of Theorem \ref{thm-eqdif1}, and so the corollary follows as $\hat{Y}_j(t)=Y_j(t)$ for $0 \leq t < T_{\hat{D}}$.
\end{proof}

\section{Proof of Theorem \ref{thm-main}}\label{sec-proof}
The proof of Theorem \ref{thm-main} is mainly based on Theorem \ref{thm-eqdif1}. Indeed we will apply this theorem to show that the trajectory of $N_i^{j, k}(t)$ and $F^{j, k}(t)$ throughout the
algorithm is a.a.s. close to the solution of the deterministic differential equations suggested
by these equations.

Define $$b(n):=\sum_{1 \leq j,k \leq n} jk (\Omega +1) .$$ For simplicity the dependence on $n$ is dropped from the notations.
For $\epsilon > 0$, we define the domains $D(\epsilon)$ as
\begin{align*}
D(\epsilon)=\{\left(\tau, \{n_i^{j,k}\}_{i< \Omega, j,k} , \{f^{j,k}\}_{j,k}  \right) \in \RR^{b(n)+1} \ : \ & \ -\epsilon < n_i^{j,k} < \lambda , -\epsilon < \tau < \lambda - \epsilon , \\ \ & \ -\epsilon < f^{j,k} < \lambda , \ \ \sum_{j,k} k f^{j,k} - \tau > 0 \:\:\:\:\:  \}.
\end{align*}

\noindent Let $T_D$ be the stopping time for $D$ which is the first time $t$ when $$\left(t/n,\{N_i^{j, k}(t/n)\},\{F^{j, k}(t/n)\}\right) \notin D .$$

\noindent Let {\rm DE} be the following system of differential equations:
\begin{eqnarray*}
(n_0^{j, k})'(\tau) &=& - \frac{j n_0^{j, k}(\tau)}{\lambda - \tau} ,\\
(n_i^{j, k})'(\tau) &=& \frac{(j-i+1) n_{i-1}^{j, k}(\tau)-(j-i) n_{i}^{j, k}(\tau)}{\lambda-\tau} , \\
(f^{j, k})'(\tau) &=& \frac{(j-\Omega+1) n_{\Omega-1}^{j, k}(\tau)}{\lambda - \tau} .
\end{eqnarray*}
with initial conditions
\begin{align*}
&n_0^{j, k} = (1-\alpha) P(j,k) , \ \ n_i^{j, k}(0) = 0 \:\: \textrm{ for } 0<i<\Omega, & f^{j, k}(0) = \alpha P(j,k) .
\end{align*}
We have

\begin{lemma} \label{lem-DE}
\begin{enumerate}
\item The system {\rm DE} has a unique solution in $D(\epsilon)$ which extends to points arbitrarily close to the boundary of $D(\epsilon)$.
\item For a sufficiently large constant $C$, with high probability we have
\begin{eqnarray}
\label{eq:diffmethod}
N_i^{j, k}(t)/n &=&  n_i^{j, k}(t/n)+o(1), \\
F^{j, k}(t)/n &=& f^{j, k}(t/n)+o(1),
\end{eqnarray}
uniformly for all $t \leq n\sigma$. Here
$\sigma=\sigma(n)$ is the supremum of those $\tau$ for which the solution of these differential
equations can be extended before reaching within $l^{\infty}$-distance $Cn^{-1/4}$ of the boundary of $D(\epsilon)$.
\end{enumerate}
\end{lemma}

\begin{proof}
We will use Theorem~\ref{thm-eqdif1}. The domain $D(\epsilon)$ is a bounded open set which contains all initial values of variables
which may happen with positive probability. Each variable is bounded by a constant times $n$.
By the definition of our process, the Boundedness Hypothesis is satisfied with $\beta(n) = 1$.
Trend Hypothesis is satisfied by some $\lambda_1(n)=O(1/n)$. Finally the third condition (Lipschitz Hypothesis) of the theorem is also satisfied since $\lambda-\tau$ is bounded away from zero. Then we set $\lambda = O(n^{-1/4}) > \lambda_1$. The conclusion of Theorem \ref{thm-eqdif1} now gives
\begin{eqnarray*}
N_i^{j, k}(t)/n &=&  n_i^{j, k}(t/n)+O(n^{3/4}), \\
F^{j, k}(t)/n &=& f^{j, k}(t/n)+O(n^{3/4}),
\end{eqnarray*}
with probability $1 - O(n^{7/4} \exp (-n^{1/4}))$ uniformly for all $t \leq n\sigma$.
Finally, for $0<i<\Omega$; we have $N_i^{j, k}(0)/n = 0$, and by Condition \ref{cond} and by definition:
\begin{align*}
N_0^{j, k}(0)/n \stackrel{p}{\rightarrow} (1-\alpha) P(j,k) , \  \: & F^{j, k}(0)/n \stackrel{p}{\rightarrow} \alpha  P(j,k) .
\end{align*}
This completes the proof.
\end{proof}

To analyze $\sigma$, we need to determine which constraint is violated when the solution
reaches the boundary of $D(\epsilon)$. It cannot be the first constraint, because (\ref{eq:diffmethod}) must give asymptotically feasible values of $N_i^{j,k}$ and $F^{j,k}$ up until the boundary is approached. It remains to
determine which of the last two constraints is violated when $\tau=\sigma$.

We first solve the system of differential equations (DE) and then we analyze the point to which the resulting equations are valid.
\begin{lemma}\label{lem-sol}
The solution of the system of differential equations (DE) is
\begin{eqnarray*}
 n_i^{j, k}(\tau) &=& P(j,k) (1-\alpha) {j\choose i} y^{j-i} (1-y)^i , \\
 f^{j, k}(\tau) &=& P(j,k) \left[ \alpha + (1-\alpha) \PP(Bin(j,1-y)\geq \Omega) \right],
\end{eqnarray*}
where $y=(1-\tau/\lambda)$.
\end{lemma}
\begin{proof}
Let $u = u(\tau)=- ln (\lambda - \tau)$. Then $u(0) = - ln (\lambda) $, $u$ is strictly monotone and so is the inverse function $\tau=\tau(u)$. 
We write the system of differential equations (DE) with respect to $u$:
\begin{eqnarray*}
(n_0^{j, k})'(u) &=& -j n_0^{j, k}(u),\\
(n_i^{j, k})'(u)&=& (j-i+1)n_{i-1}^{j, k}(u) - (j-i)n_i^{j, k}(u).
\end{eqnarray*}
Then using
$$\frac{d}{du} (n_i^{j, k}(u)e^{(j-i-1)(u-u(0))}) = e^{(j-i-1)(u-u(0))} (j-i) n_i^{j, k}(u) ,$$ and by induction, we find
$$n_i^{j, k}(u) = e^{-(j-i)(u - u(0))}\sum_{r=0}^i {{j-r}\choose{i-r}}\left( 1 - e^{-(u-u(0))} \right)^{i-r} n_i^{j, k}(u(0)).$$
By going back to $\tau$, we have
$$n_i^{j, k}(\tau) = y^{d-j} \sum_{r=0}^i n_i^{j, k}(0) {{j-r}\choose{i-r}}(1-y)^{i-r}, \  y=(1-\tau/\lambda),$$ which gives
$$n_i^{j, k}(\tau) = P(j,k) (1-\alpha) {j\choose i} y^{j-i} (1-y)^i .$$
We have
\begin{eqnarray*}
(f^{j, k})'(y) &=& - \lambda (f^{j, k})'(\tau) \\
&=& - \lambda  \frac{(j-\Omega+1)n_{\Omega-1}^{j, k}}{\lambda y} \\
&=& - (j-\Omega+1) P(j,k) (1-\alpha) {j\choose \Omega-1} y^{j-\Omega} (1-y)^{\Omega-1} \\
&=& - P(j,k) (1-\alpha) j\PP(Bin(j-1,1-y) = \Omega-1).
\end{eqnarray*}
Then using the fact that
$$ \frac{\partial}{\partial p} \mathbb{P}(Bin(N, p) > K) = N \mathbb{P}(Bin(N-1, p) = K), $$
and by initial condition we have
$$f^{j, k} =  P(j,k) \left[ \alpha + (1-\alpha) \PP(Bin(j,1-y)\geq \Omega) \right] .$$
\end{proof}

Let us define
\begin{eqnarray}
f_{out}(\tau) &:=&  \sum_{k,j} k f^{j, k}(\tau) - \tau \label{eq-fout}, \mbox {and} \\
f(\tau) &=& \sum_{k,j} k f^{j, k}(\tau).
\end{eqnarray}
\begin{lemma}\label{lem-converge}
Assume $\sigma=\sigma(n)$ be the same as Lemma \ref{lem-DE}. For $t \leq n \sigma$, we have
\begin{eqnarray}
\left| F_{out}(t)/n - f_{out}(t/n) \right| \stackrel{p}{\rightarrow} 0 , \label{eq:cvg-out}
\end{eqnarray}
and
\begin{eqnarray}
\left| F(t)/n - f(t/n) \right| \stackrel{p}{\rightarrow} 0 .
\end{eqnarray}
\end{lemma}
\begin{proof}
By definition, we have
\begin{eqnarray*}
\left| F_{out}(t)/n - f_{out}(t/n) \right| &=& \left| \sum_{j,k}k \left(F^{j,k}(t)/n-f^{j, k}(t/n)\right) \right| \\
&\leq& \sum_{j,k} k \left| F^{j,k}(t)/n-f^{j, k}(t/n) \right| ,
\end{eqnarray*}
and
\begin{eqnarray*}
\left| F(t)/n - f(t/n) \right| &\leq& \sum_{j,k} \left| F^{j,k}(t)/n-f^{j, k}(t/n) \right| \\
&\leq& \sum_{j,k} k \left| F^{j,k}(t)/n-f^{j, k}(t/n) \right| .
\end{eqnarray*}
It remains to show $$\sum_{j,k}k \left|F^{j,k}(t)/n-f^{j, k}(t/n)\right| \stackrel{p}{\rightarrow} 0 .$$
By Lemma \ref{lem-DE}, for each $j$ and $k$, we have $\left|F^{j,k}(t)/n-f^{j, k}(t/n)\right| \stackrel{p}{\rightarrow} 0$. Hence, the same holds for any finite partial sum, which is for each $K\in \NN$, we have $$\sum_{j,k<K} j \left|F^{j,k}(t)/n-f^{j, k}(t/n)\right| \stackrel{p}{\rightarrow} 0  .$$ By Condition \ref{cond}, $\sum_{j,k} kP(j,k) \rightarrow \lambda \in (0,\infty)$. Then, there exist a constant $K$, such that $\sum_{j, k\geq K} kP(j,k) < \epsilon $. By Lemma \ref{lem-sol}
$$f^{j, k}(\tau) =  P(j,k) \left[ \alpha + (1-\alpha) \PP(Bin(j,1-y)\geq \Omega) \right] \leq P(j,k) .$$ Let $N(j,k)$ denote the number of vertices with in-degree $j$ and out-degree $k$ at time $0$. Again, by Condition \ref{cond}, $\sum_{j,k} k N(j,k)/n \rightarrow k P(j,k)$. Hence, also $\sum_{j,k\geq K} kN(j,k)/n \rightarrow \sum_{j,k\geq K} kP(j,k) < \epsilon$. Therefore, if $n$ is large enough, $\sum_{j, k\geq K} k N(j,k)/n < \epsilon$, and
\begin{eqnarray*}
\sum_{j, k \geq K} k \left|F^{j,k}(t)/n-f^{j, k}(t/n)\right| &\leq&  \sum_{j, k \geq K} k  \left( F^{j, k}(t)/n + f^{j, k} (t/n) \right) \\ &\leq&  \sum_{j, k \geq K} k \left(N(j,k)/n + P(j,k) \right) < 2\epsilon . \end{eqnarray*}
This completes the proof.
\end{proof}

We now return to the proof of Theorem \ref{thm-main}. By Lemma \ref{lem-sol} and by equation (\ref{eq-fout}), we have
\begin{eqnarray*}
f_{out}(\tau)  &=& \sum_{k,j} k P(j,k) \left[ \alpha + (1-\alpha) \PP(Bin(j,1-y)\geq \Omega) \right] - \tau \\
&=& \lambda \alpha  +  \sum_{k,j} k P(j,k) (1-\alpha) \PP(Bin(j,1-y)\geq \Omega) - \tau \\
&=&  \lambda y - \lambda(1-\alpha) + (1-\alpha) \\ && \mathbb{E}\left[ D_{out} \ind\left(Bin(D_{in},1-y)\geq \Omega\right)\right] \\ &=& f_{\alpha}(y),
\end{eqnarray*}
where $y=(1-\tau/\lambda)$.

\noindent First assume  $f_{\alpha}(y) > 0$ for all $y \in (0,1]$, i.e., $y^*=0$. This implies $D_{in} \geq \Omega$ almost surely. Then we have $\sum_{j,k} k f^{j,k} - \tau > 0$ in $D(\epsilon)$. So the boundary reached is determined by $\hat{\tau}=\lambda - \epsilon$ which is $\hat{y}=\epsilon/\lambda$. Then it is easy to see $f_{out}(\hat{\tau})=O(\epsilon)$  and Lemma \ref{lem-converge} implies $|A_f|=n-O(n\epsilon)$. This proves the first part of the theorem.\\

\noindent Now consider $y^*>0$, and further $y^*$ is not a local minimum point of $f_{\alpha}(y)$. Then $f_{\alpha}(y)<0$
for some interval $(y^*-\epsilon,y^*)$. Then the first constraint is violated at time $\hat{\tau}=\lambda(1- y^*)$.
We apply Corollary \ref{cor-eqdif2} with $\hat{D}$ the domain $D(\epsilon)$ defined above, and the
domain $D$ replaced by $D'(\epsilon)$, which is the same as $D$ except that the last constraint
is omitted:
\begin{align*}
D'(\epsilon)=\{\left(\tau, \{n_i^{j,k}\}_{i< \Omega, j,k} , \{f^{j,k}\}_{j,k}  \right) \in \RR^{b(n)+1} \ : \ & \ -\epsilon < n_i^{j,k} < \lambda , -\epsilon < \tau < \lambda - \epsilon , \\ \ & \ -\epsilon < f^{j,k} < \lambda \:\:\:\:\:  \}.
\end{align*}

This gives us the convergence in equations (\ref{eq:diffmethod}) upto the point where the solution leaves
$D'(\epsilon)$ or when $F_{out}(t)>0$ is violated. Since $f_{out}(\tau)$ begins to go negative after $\hat{\tau}$,
from equation (\ref{eq:cvg-out}) it follows that $F_{out}(t)>0$ must be violated a.a.s. and it becomes zero at some $T_f \sim \hat{\tau} n$.
Then by Lemma \ref{lem-sol} and Lemma \ref{lem-converge} we conclude
\begin{eqnarray*}
F(T_f) &=&  \sum_{k,j} F^{j, k}(T_f) \\ &=& n \left( 1 - (1-\alpha)\mathbb{E}\left[\ind\left(Bin(D_{in},1-y^*)<\Omega \right)\right] \right)  + o_p(n).
\end{eqnarray*}
This completes the proof.

\section*{Acknowledgements} I would like to thank Fran\c{c}ois Baccelli, Marc Lelarge, Andreea Minca and Tsvi Tlusty for helpful comments and discussions.

\bibliographystyle{plain}
\footnotesize
\bibliography{neural}

\end{document}